\documentclass[12pt]{article}

\usepackage{amsfonts}
\usepackage{amssymb}
\usepackage{amsthm}
\usepackage[centertags]{amsmath}
\usepackage{newlfont}
\usepackage{graphicx}

\newfont{\bb}{msbm10 at 12pt}
\def\r{\hbox{\bb R}}

\def\e{\hbox{\bf E}}
\def\t{\hbox{\bf T}}

\def\b{\hbox{\bf B}}

\newtheorem{theorem}{Theorem}[section]
\newtheorem{definition}[theorem]{Definition}
\newtheorem{proposition}[theorem]{Proposition}
\newtheorem{remark}[theorem]{Remark}

\begin{document}

\title{Helicoidal surfaces in Minkowski space with constant mean  curvature  and constant Gauss curvature}
\author{Rafael L\'opez\footnote{Corresponding author. The first author is partially
supported by MEC-FEDER
 grant no. MTM2007-61775 and
Junta de Andaluc\'{\i}a grant no. P06-FQM-01642.}\\
Departamento de Geometr\'{\i}a y Topolog\'{\i}a\\
Universidad de Granada\\
18071 Granada, Spain \\
email: rcamino@ugr.es\\
\vspace*{.5cm}\\
 Esma Demir  \\
  Department of Mathematics\\
University of Ankara\\
Tandon\u{g}an, Ankara 06100, Turkey}

\date{}

\maketitle

\begin{abstract}  In this work we find all helicoidal surfaces in Minkowski space  with constant mean curvature
whose generating curve is a the graph of a polynomial or   a Lorentzian circle. In the first case, we prove that  the degree of the polynomial is $0$ or $1$ and that the surface is ruled. If the generating curve is a Lorentzian circle, we show that the only possibility is that the axis is spacelike and the center of the circle lies in the axis.
\end{abstract}



\section{Introduction and statement of results}

Consider the Lorentz-Minkowski space $\e_1^3$, that is, the three-dimensional real vector space $\r^3$ endowed the metric $\langle,\rangle$ given by
$$\langle (x,y,z),(x',y',z')\rangle=xx'+yy'-zz',$$
where $(x,y,z)$ are the usual coordinates of $\r^3$. A Lorentzian motion of $\e_1^3$ is a Lorentzian rotation around an axis $L$ followed by a translation.  A helicoidal surface in Minkowski space $\e_1^3$ is a surface invariant by a uniparametric group $G_{L,h}=\{\phi_t;t\in\r\}$ of helicoidal motions. Each group of helicoidal motions is characterized by  an axis $L$ and a pitch  $h\not=0$ and each helicoidal surface is given by a group of helicoidal motions and a generating curve $\gamma$. In particular,  a helicoidal surface parametrizes as $X(s,t)=\phi_t(\gamma(s))$, $t\in\r$, $s\in I\subset\r$.

The first part of this work is motivated by the results that appear in   \cite{dk}. The authors study in \cite{dk} helicoidal surfaces generated by a straight-line (called \emph{helicoidal ruled surfaces}). Among the examples, we point out the surfaces named \emph{right Lorentzian helicoids} by Dillen and K\"{u}hnel, that is, the axis $L$ is timelike or spacelike and the curve $\gamma$ is one of the coordinate axis.  These surfaces are the helicoid of first kind (if $L$ is timelike), the helicoid of second type (if $L$ is spacelike and  $\gamma$ is the $y$-axis) and the helicoid of third type (if $L$ is spacelike and  $\gamma$ is the $z$-axis). These three surfaces have zero mean curvature. When the axis is lightlike, there are two known helicoidal surfaces generated by a straight-lines and
called in the literature  Lie's minimal surface or  Cayley's surface and the parabolic null cylinder (\cite{dk,mp,st}). All these surfaces have zero mean curvature.  In this paper we consider a generalization of this setting.  In fact, we take $\gamma$  the graph of a polynomial $f(s)=\sum_{n=0}^m a_n s^n$ and we ask if the corresponding helicoidal surface has constant mean curvature.   We prove in Section \ref{sectt1}:

\begin{theorem}\label{t1} Consider a  helicoidal surface in $\e_1^3$ with constant mean curvature $H$ whose generating curve is the graph of a polynomial $f(s)=\sum_{n=0}^m a_n s^n$. Then $m\leq 1$, that is, the generating curve is a
 straight-line. Moreover, and after a rigid motion of $\e_1^3$,
\begin{enumerate}
\item If the axis is timelike $L=<(0,0,1)>$,   the surface is the helicoid of first kind ($H=0$), the surface    $X(s,t)=(s\cos{(t)},s\sin({t}),\pm s+a_0 +ht)$, $a_0\in\r$ with  $|H|=1/h$ or the Lorentzian cylinder $x^2+y^2=r^2$ with $|H|=1/(2r)$.
\item If the axis is spacelike $L=<(1,0,0)>$, then $H=0$. The surface is the helicoid of second kind, the helicoid of third kind or the surface parametrized by $X(s,t)=(ht,(\pm s+a_0)\sinh{(t)}+s\cosh{(t)},(\pm s+a_0)\cosh{(t)}+s\sinh{(t)})$,  $a_0\not=0$.
\item If the axis is lightlike $L=<(1,0,1)>$, then $H=0$ and the surface is the Cayley's surface or the parabolic null cylinder.
\end{enumerate}
\end{theorem}

In   Section \ref{sectt1}, we will also study helicoidal surfaces where $H^2-K=0$. Recall that in Minkowski space, there are non-umbilical timelike surfaces where $H^2-K=0$. We will find all such surfaces when the generating curve is the graph of a polynomial.

The motivation of the second part of this article comes from helicoidal surfaces whose generating curve is a Lorentzian circle of $\e_1^3$. For example, we consider the curve
$\gamma(s)=(0,r\cosh{(s)},r\sinh{(s)})$, $r>0$, and we apply a group of helicoidal motions whose axis is
$L=<(1,0,0)>$. The corresponding surface  is the timelike hyperbolic cylinder $y^2-z^2=r^2$ with $|H|=1/(2r)$. Similarly, one can take the curve   $\gamma(s)=(0,r\sinh{(s)},r\cosh{(s)})$, obtaining the spacelike hyperbolic cylinder  $y^2-z^2=-r^2$. In this case   $|H|=1/(2r)$ again. In \cite{mp}, the authors call {\it right circular cylinders} those helicoidal surfaces generated by circles. From these examples,
 we consider the problem of finding all  helicoidal surfaces with constant mean curvature whose generating curve is a Lorentzian circle of $\r^3$, and we ask if the above examples are the only ones possible. We conclude:

\begin{theorem}\label{t2} Consider a helicoidal surface in $\e_1^3$ with constant mean curvature $H$ whose generating curve is a Lorentzian circle of $\e_1^3$. Then the axis of the surface is spacelike  and $H\not=0$. Moreover the center of the circle lies in the axis and, up a rigid motion of $\e_1^3$, the surface is one the hyperbolic cylinders  $y^2-z^2=\pm r^2$.
\end{theorem}

We end  studying helicoidal surfaces with constant Gauss curvature. When the axis is timelike, the second surface in Theorem \ref{t1}   has $K=1/h^2$. On the other hand, the examples of Theorem \ref{t2} satisfy $K=0$. We show that they are the only possible under the same conditions as in Theorems \ref{t1} and \ref{t2}.

\begin{theorem}\label{t3} Consider a helicoidal surface in $\e_1^3$ with constant Gauss curvature $K$.
\begin{enumerate}
\item If the generating curve is the graph of a polynomial $f(s)=\sum_{n=0}^m a_n s^n$, then $m\leq 1$. If
the axis is timelike, the surface is the Lorentzian cylinder $x^2+y^2=r^2$ ($K=0$) or
 the surface $X(s,t)=(s\cos{(t)},s\sin({t}),\pm s+a_0 +ht)$ with  $K=1/h^2$; if the axis is spacelike, the surface is $X(s,t)=(ht,(\pm s+a_0)\sinh{(t)}+s\cosh{(t)},(\pm s+a_0)\cosh{(t)}+s\sinh{(t)})$,  $a_0\not=0$ ($K=0$); if the axis is lightlike, the surface is the parabolic null cylinder ($K=0$).
\item  If the generating curve is a circle, then   the axis  is spacelike, $K=0$,  the center of the circle lies in the axis and the  surface is one of the hyperbolic cylinders $y^2-z^2=\pm r^2$.
    \end{enumerate}
\end{theorem}

Throughout this work, we will assume that the helicoidal motions are not rotational, that is, $h\not=0$. Rotational surfaces with constant mean curvature and constant Gauss curvature have been studied in \cite{hn1,hn2,ko,lo1,lo2}.

This article is organized beginning with the Section \ref{sectd} where we present the parametrizations of helicoidal surfaces as well as the definition of a Lorentzian circle in $\e_1^3$. Next in Section \ref{sectc} we recall the definition of the  mean curvature and the Gauss curvature of a non-degenerate surface, describing the way to compute in local coordinates. The rest of the article is the proof of the results, beginning in Section \ref{sectt1} with the Theorem \ref{t1}, and following with Sections \ref{sectt2} and \ref{sectt3} with   Theorems \ref{t2} and \ref{t3}, respectively.

{\it Acknowledgement}. This work was done during the stay of the second author in the Department of Geometry and Topology, in the University of Granada, during May and June of 2010. This article is part of her Master Thesis, whose advisors are Prof. Yayli (University of Ankara) and the first author.

\section{Description of helicoidal surfaces of $\e_1^3$}\label{sectd}

In this section we describe the parametrization of a helicoidal surface in $\e_1^3$ and we recall the notion of a Lorentzian circle. The metric $\langle,\rangle$ of  $\e_1^3$ divides the vectors in three types according its causal character. A  vector  $v\in\e_1^3$ is called spacelike (resp. timelike, lightlike) if  $\langle v,v\rangle>0$ or $v=0$ (resp. $\langle v,v\rangle<0$, $\langle v,v\rangle=0$ and  $v\not=0$). Given  a vector subspace $U\subset\e_1^3$, we say that $U$ is called spacelike (resp. timelike, lightlike)  if the induced metric  is positive definite (resp. non-degenerate of index $1$, degenerated and $U\not=\{0\}$). The description of Lorentzian motion groups is the following

\begin{proposition}\label{pr1} A helicoidal Lorentzian motion group is a uniparametric group of Lorentzian rigid motions which are non trivial. Any  group of helicoidal motions group is determined by an axis $L$ and a pitch $h\in\r$, which it will be denoted by  $G_{L,h}=\{\phi_t;t\in\r\}$. After a change of coordinates any helicoidal motions group is given by:
\begin{enumerate}
\item If $L$ is timelike, then $L=<(0,0,1)>$ and
\begin{equation}\label{motion1}
\phi_t(a,b,c)=\left(\begin{array}{ccc}
\cos{t}&-\sin{t}&0\\ \sin{t}&\cos{t}&0\\
0& 0&1\end{array}\right)\left(\begin{array}{c}a\\b\\c\end{array}\right)+h\left(\begin{array}{c}0\\0\\t\end{array}\right).
\end{equation}
\item If $L$ is spacelike, then $L=<(1,0,0)>$ and
\begin{equation}\label{motion2} \phi_t(a,b,c)=\left(\begin{array}{ccc}1&0&0\\
0&\cosh{t}&\sinh{t}\\
0&\sinh{t}&\cosh{t}\end{array}\right)\left(\begin{array}{c}a\\b\\c\end{array}\right)+
h\left(\begin{array}{c}t\\0\\0\end{array}\right).
\end{equation}
\item If $L$ is lightlike, then $L=<(1,0,1)>$ and
\begin{equation}\label{motion3}
\phi_t(a,b,c)=\left(
\begin{array}{ccc}
1-\frac{t^2}{2}&t &\frac{t^2}{2}\\
-t &1&t\\
-\frac{t^2}{2} & t &1+\frac{t^2}{2}
\end{array}
\right)\left(\begin{array}{c}a\\b\\c\end{array}\right)+
h\left(\begin{array}{c}\frac{t^3}{3}-t\\ t^2\\ \frac{t^3}{3}+t\end{array}\right).
\end{equation}
\end{enumerate}
If we take $h=0$, then we obtain a rotations group about the axis $L$.
\end{proposition}

If the axis is spacelike or timelike, the translation vector is the direction of the axis.
  The following result is obtained   in \cite[Lemma 2.1]{mp} and it says how to parametrize a helicoidal surface.

\begin{proposition}\label{pr2} Let $S$ be a surface in $\e_1^3$ invariant by a group of helicoidal motions $G_{L,h}=\{\phi_t;t\in\r\}$. Then  there exists a planar curve $\gamma=\gamma(s)$ such that $S=\{\phi_t(\gamma(s));s\in I,t\in\r\}$.  The curve $\gamma$ is called a generating curve of $S$. Moreover,
 \begin{enumerate}
\item if $L$ is timelike, $\gamma$ lies in any plane containing $L$.
\item if $L $ is spacelike, then $\gamma$ lies in a orthogonal plane to $L$.
\item  if $L$ is spacelike, $\gamma$ lies in the   only degenerate plane containing $L$.
\end{enumerate}
\end{proposition}

Thus, by Propositions \ref{pr1} and \ref{pr2},   any helicoidal surface in $\e_1^3$ locally parametrizes  as:
\begin{enumerate}
\item If the axis is timelike, with $L=<(0,0,1)>$, and $\gamma(s)=(s,0,f(s))$, then
\begin{equation}\label{p-1}
X(s,t)=(s\cos{(t)},s\sin{(t)},ht+f(s)),\ s\in I, t\in\r.\end{equation}
\item If the axis is spacelike,  with $L=<(1,0,0)>$,  and $\gamma(s)=(0,s,f(s))$, then
\begin{equation}\label{p-2}
 X(s,t)=(ht,s\cosh{(t)}+f(s)\sinh{(t)},s\sinh{(t)}+f(s)\cosh{(t)}),\ s\in I, t\in\r.
 \end{equation}
\item If the axis is lightlike, with $L=<(1,0,1)>$,  and $\gamma(s)=(f(s),s,f(s))$, then
\begin{equation}\label{p-3}
X(s,t)=(st+ h(\frac{t^3}{t}-t)+f(s),s+ht^2,st+h(\frac{t^3}{3}+t)+f(s)).\end{equation}
 \end{enumerate}

\begin{remark} In \cite{dk} the authors define a ruled helicoidal surface as a helicoidal surface generating by a straight-line. It is evident that  any  ruled helicoidal surface is both a ruled surface and a helicoidal surface. However, there are ruled surfaces that are helicoidal but they are not generated by a straight-line in the sense of Proposition \ref{pr2}. For example,  the timelike hyperbolic cylinder  $y^2-z^2=r^2$ is helicoidal whose axis is $L=<(1,0,0)>$, and   it is also a ruled surface, but the intersection of the surface with the plane $x=0$ is $\gamma(s)=(0,r\cosh{(s)},r\sinh{(s)})$, which it is not a  straight-line. In fact, the surface is invariant for all helicoidal motions with axis $L$ and arbitrary pitch $h$ \cite{mp}. On the other hand, this surface can been viewed as a  surface of revolution with axis $L$ obtained rotating the curve  $\alpha(s)=(s,0,r)$. This curve $\alpha$ is not a generating curve according to Proposition \ref{pr2}.
\end{remark}

Finally we end with the definition of a   circle in $\e_1^3$(see \cite{lls}).

\begin{definition} A Lorentzian circle in $\e_1^3$ if the orbit of a point under a group of rotations.
\end{definition}

Let $p=(a,b,c) $ be a point of $\e_1^3$ and let $G_L=\{\phi_t;t\in\r\}$ a group of rotations with axis $L$. We are going to describe the trajectory of $p$ by $G_L$, that is, $\alpha(t)=\phi_t(p)$, $t\in\r$. We assume that $p\not\in L$ because in such case $\alpha$ is a point. Depending on the   causal character of $L$, we have three cases.

\begin{enumerate}
\item The axis is timelike,  $L=<(0,0,1)>$.   Then
    $\alpha(t)=(a \cos{(t)}-b\sin{(t)},b\cos{(t)}+a\sin{(t)},c)$. This curve is an Euclidean circle of radius $\sqrt{a^2+b^2}$ contained in the plane $z=c$.
\item The axis is spacelike, $L=<(1,0,0)>$.  Now $\alpha(t)=(a,b\cosh{(t)}+c\sinh{(t)},c\cosh{(t)}+b\sinh{(t)})$ with $|\alpha'(t)|^2=-b^2+c^2$. Depending on the relation between $b$ and $c$, we distinguish three sub-cases:
    \begin{enumerate}
    \item If $b^2<c^2$, $\alpha$ is spacelike and it intersects the $z$-axis in one point. After a translation, we assume that $p=(0,0,c)$. Then $\alpha(t)=(0,c\sinh{(t)},c\cosh{(t)})$. This curve is the hyperbola $z^2-y^2=c^2$ in the plane $x=0$.
    \item If $b^2=c^2$, then $\alpha$ is lightlike, $\alpha(t)=(a,\pm c(\cosh{(t)}+\sinh{(t)}),c(\cosh{(t)}+\sinh{(t)})$. Thus $\alpha$ is one of the straight-lines $y=\pm z$ in the plane $x=a$.
    \item If $b^2>c^2$, $\alpha$ is timelike and it intersects  the $y$-axis in one point. Now we suppose that $p=(0,b,0)$ and so $\alpha(t)=(0,b\cosh{(t)},b\sinh{(t)})$. This curve is the hyperbola $y^2-z^2=b^2$ in the plane $x=0$.
    \end{enumerate}
\item The axis is lightlike, $L=<(1,0,1)>$ and that $p=(a,0,c)$.  Because $|\alpha'(t)|^2=(a-c)^2$ and $p\not\in L$, then $\alpha$ is a spacelike curve:   $\alpha(t)=(a,0,c)+(c-a)t (0,1,0)+(c-a)/2 t^2(1,0,1)$. This curve lies in the plane $x-z=a-c$ and from the Euclidean viewpoint, this curve is a parabola with axis parallel to $(1,0,1)$.
\end{enumerate}

\section{Curvature of a non-degenerate surface}\label{sectc}

Part of this section can seen in   \cite{on,we}. An immersion $x:M\rightarrow \e_1^3$  of a surface $M$ is called spacelike (resp. timelike) if the tangent plane $T_pM$  is spacelike (resp. timelike) for all $p\in M$.
We also say that $M$ is spacelike (resp, timelike). In both cases, we say that the surface is non-degenerate.
We define the mean curvature $H$ and the Gauss curvature $K$  of a non-degenerate surface. For this, let $\mathfrak{X}(M)$ be the set of tangent vector fields to $M$.  We denote by  $\nabla^0$ the Levi-Civitta connection of $\e_1^3$ and $\nabla$ the induced connection on $M$ by the immersion $x$, that is, $\nabla_X Y=(\nabla_X^0 Y)^\top$, where $\top$ denotes the tangent part of the vector field $\nabla^0_X Y$. We have the decomposition
\begin{equation}\label{3-gf}
\nabla_X^0 Y=\nabla_X Y+\sigma(X,Y),
\end{equation}
called the Gauss formula. Here $\sigma(X,Y)$ is the normal part of the vector $\nabla^0_XY$.
Now consider $\xi$ a normal vector field to  $x$ and we do $-\nabla_X^0 \xi$. We denote by $A_\xi (X)$ its tangent component, that is, $A_\xi(X)=-(\nabla_X^0 \xi)^\top$. We have from (\ref{3-gf}) that
\begin{equation}\label{3-symm}
\langle A_\xi(X),Y\rangle=\langle\sigma(X,Y),\xi\rangle.
\end{equation}
The map  $A_\xi:\mathfrak{X}(M)\rightarrow\mathfrak{X}(M)$ is called the Weingarten endomorphism associated to $\xi$. Because $\sigma$ is symmetric, we have from (\ref{3-symm}) that
\begin{equation}\label{3-self}
\langle A_\xi(X),Y\rangle=\langle X,A_\xi(Y)\rangle.
\end{equation}
This means that the map $A_\xi$  is linear and self-adjoint with respect to the metric of $M$.
From now, we assume that the surface is orientable (in fact, any spacelike  surface is  orientable). In the case that the surface is timelike, we assume that it is locally orientable. Denote by $N$ a Gauss map   on $M$. Define
$\epsilon$ by $\langle N,N\rangle=\epsilon$, where $\epsilon=-1$ (resp. $1)$ if the immersion is spacelike (resp. timelike). If we take   $\xi=N$, and because  $\langle N,N\rangle=\epsilon$,
we have $\langle\nabla_X^0 N,N\rangle=0$. Then the normal part of $\nabla_XN$ vanishes and we obtain the Weingarten formula
\begin{equation}\label{3-dn}
-\nabla_X^0 N=A_N(X).
\end{equation}

\begin{definition}
The Weingarten endomorphism at $p\in M$ is defined by  $A_p:T_pM\rightarrow T_pM$, $A_p=A_{N(p)}$, that is,
if $v\in T_pM$, let $X\in\mathfrak{X}(M)$ be a tangent vector field that extends $v$, then $A_p(v)=(A(X))_p$.
 Moreover, from (\ref{3-dn})
 $$A_p(v)=-(dN)_p(v),\ \ v\in T_pM,$$
 where $(dN)_p$ is the usual differentiation in $\e_1^3$ of the map $N$ at $p$.
 \end{definition}

 Because $\sigma(X,Y)$ is proportional to $N$, we have from  (\ref{3-gf}) and (\ref{3-symm}) that
\begin{equation}\label{3-ga2}
\sigma(X,Y)=\epsilon\langle\sigma(X,Y),N\rangle N=\epsilon\langle A(X),Y\rangle N.
\end{equation}
Now (\ref{3-gf}) writes as $\nabla_X^0 Y=\nabla_X Y+\epsilon\langle A(X),Y\rangle N$.

\begin{definition}
Given a non-degenerate immersion,  the  mean curvature vector field $\vec{H}$ and the Gauss curvature $K$ are
$$\vec{H}=\frac12\mbox{trace}_I(\sigma),\hspace*{1cm}K=\epsilon\frac{\mbox{det}(\sigma)}{\mbox{det}(I)},$$
where the subscript $\mbox{I}$ means that the computation is done with respect to the metric $\mbox{I}=\langle,\rangle$. The mean curvature function $H$ is given by $\vec{H}=HN$, that is,  $H=\epsilon \langle\vec{H},N\rangle$.
\end{definition}
In terms of the Weingarten endomorphism $A$,  $H$ and $K$ are
$$H=\frac{\epsilon}{2}\mbox{trace}(A),\hspace*{1cm}K=\epsilon\mbox{det}(A).$$

In this work we need to compute $H$ and $K$  using a parametrization of the surface. Let  $X:U\subset\r^2\rightarrow\e_1^3$ be a parametrization of the surface, $X=X(u,v)$.
Denote $\mbox{II}(w_1,w_2)=\langle Aw_1,w_2\rangle$, with $w_i\in T_{X(u,v)}M$. Then $A=\mbox{(II)}(\mbox{I})^{-1}$. Fix the basis $B$ of the tangent plane given by
$$X_u:=\frac{\partial X(u,v)}{\partial u},\hspace*{1cm}X_v:=\frac{\partial X(u,v)}{\partial v}.$$
We define $\{E,F,G\}$ and $\{e,f,g\}$   the coefficients of $\mbox{I}$ and $\mbox{II}$ with respect to $B$, respectively:
$$E=\langle X_u,X_u\rangle,\ F=\langle X_u,X_v\rangle,\ G=\langle X_v,x_v\rangle,$$
$$e=-\langle N_u,X_{u}\rangle,\ f=-\langle N_u,X_{v}\rangle,\ g=-\langle N_v,X_{v}\rangle,$$
Then
\begin{equation}\label{hk}
H=\epsilon\frac12\frac{eG-2fF+gE}{EG-F^2},\hspace*{1cm}K=\epsilon\frac{eg-f^2}{EG-F^2}.\end{equation}
Here $N$ is
$$N=\frac{X_u\times X_v}{\sqrt{-\epsilon(EG-F^2)}}.$$
 We recall that $W:=EG-F^2$ is positive (resp. negative) if the immersion is spacelike (resp. timelike).
Finally, in order to do the computations for $H$ and $K$, we recall that the cross-product $\times$ satisfies that
for any vectors $u,v,w\in\e_1^3$, $\langle u\times v,w\rangle=\mbox{det}(u,v,w)$. Then
(\ref{hk}) writes as
\begin{eqnarray}
H&=&-\frac{1}{2}\frac{G\mbox{det}(X_u,X_v,X_{uu})-
2F\mbox{det}(X_u,X_v,X_{uv})+E\mbox{det}(X_u,X_v,X_{vv})}{(-\epsilon(EG-F^2))^{3/2}}\nonumber \\
&:=&-\frac12\frac{H_1}{(-\epsilon(EG-F^2))^{3/2}}\label{h1}\\
K&=&-\frac{\mbox{det}(X_u,X_v,X_{uu})\mbox{det}(X_u,X_v,X_{vv})-
\mbox{det}(X_u,X_v,X_{uv})^2}{ (EG-F^2)^2}\nonumber \\
&:=&-\frac{K_1}{ (EG-F^2)^2}.\label{hk22}
\end{eqnarray}
From (\ref{h1}),  we have
\begin{equation}\label{h2}
4H^2 |EG-F^2|^3-H_1^2=0.\end{equation}

\section{Proof of Theorem \ref{t1}}\label{sectt1}
We consider a helicoidal surface generated by   the graph of the polynomial $f(s)=\sum_{n=0}^m a_n s^n$, with $a_m\not=0$. We distinguish cases according to the causal character of the axis.

\subsection{The axis is timelike }

  Assume that $L=<(0,0,1)>$. By Proposition \ref{pr2}, we suppose that the generating curve
  $\gamma$ is contained in the plane $y=0$. If $\gamma$ is not locally a graph on the $x$-axis, then
   $\gamma$ is the vertical line $\gamma(s)=(r,0,s)$, whose corresponding helicoidal surface is the
   Lorentzian cylinder $x^2+y^2=r^2$. This surface has $|H|=1/(2r)$. Thus we assume that $\gamma$  is given by $\gamma(s)=(s,0,f(s))$. The mean curvature is
$$H=-\frac12\frac{s^2f'(1-f'^2)+s(s^2-h^2)f''-2h^2f'}{(-h^2+s^2(1-f'^2))^{3/2}}.$$
We separate the cases that $H$ is zero or not zero.

If $H=0$, the numerator of the above equation  is a polynomial on $s$. Then all coefficients must vanish. If $m\geq 2$, the leader coefficient corresponds to $s^2f'^3$, that is, to $s^{3m-1}$. This coefficient is  $-m^3a_m^3$ and this implies $a_m=0$: contradiction. Thus $m<2$. If $m=0$, then $f(s)=a_0$ and  $H=0$.
 If $m=1$, the leader coefficient is $a_1(1-a_1^2)=0$. Thus $a_1=\pm 1$.  Now $H_1=\pm 2h^2$: contradiction.

Assume $H\not=0$. From (\ref{h2}) and if $m\geq 2$, the leader coefficient corresponds to $s^6 f'^6$, that is, for $s^{6m}$. The corresponding coefficient is $4H^2m^6 a_m^6$: contradiction. If $m=1$, (\ref{h2}) is a polynomial of degree $6$, with leader coefficient $4H^2(1-a_1^2)^3$. Thus $a_1=\pm 1$. With this value of $a_1$, $W=-h^2$. Now (\ref{h2}) is $4h^4(-1+h^2 H^2)$, obtaining that $|H|=1/h$.

As conclusion, we obtain  $f$ is a constant function and $H=0$ or $f(s)=\pm s+a_0$ and $H\not=0$. In the first case,
the surface $X(s,t)=(s\cos{(t)},s\sin{(t)},ht+a_0)$, which it is the helicoid of first kind followed with a translation in the direction of the axis $L$.

\subsection{The axis is spacelike }
Consider that the axis is $L=<(1,0,0)>$ and that the generating curve $\gamma$ is contained in the plane $x=0$ (Proposition \ref{pr2}). As in the timelike case, if $\gamma$ is not locally a graph on the $y$-axis, then
$\gamma(s)=(0,b,s)$. The helicoidal surface has constant mean curvature if $b=0$, with $H=0$. The surface is the helicoid of third kind. Assume  $\gamma(s)=(0,s,f(s))$. Now the mean curvature is
$$H=-\frac12\frac{h(-f+sf')(-1+f'^2)-h(h^2-s^2+f^2)f''}{(-\epsilon(h^2-s^2+2sff'-(h^2+f^2)f'^2))^{3/2}}.$$
First, let us assume that $H=0$. The numerator of this expression of $H$ vanishes for any $s$. It is  a polynomial   whose leader coefficient corresponds to $s^{3m-2}$. This coefficient  is $ha_m^3m(m-1)^2$. Thus, if $m\geq 2$, we obtain $a_m=0$, which it is a contradiction. Thus, $m\leq 1$.  If $m=0$, then $H_1=h a_0$. This means that $a_0=0$. Suppose now that
$m=1$. In such case $H=0$ is equivalent to $ha_0(1-a_1)^2=0$. We conclude that
$a_0=0$ or $a_1=\pm 1$. If $a_1=\pm 1$, then $W=-a_0^2$, and thus, $a_0\not=0$.

We suppose that $H$ is a constant with $H\not=0$. The polynomial on $s$ given by (\ref{h2}) has as leader coefficient $4H^2m^6a_m^{12}$ if $m\geq 2$, which corresponds with $s^{12m-6}$. Then $a_m=0$: contradiction. If $m=1$,
(\ref{h2}) is a polynomial of degree $6$. The corresponding coefficient is $4H^2(a_1^2-1)^6$. Then
$a_1=\pm 1$. But we know that $H=0$: contradiction.

The conclusion is  that $H=0$ and the degree of $f$ is $0$ or $1$. Exactly, the only cases are   $f(s)=a_1 s$ or $f(s)=\pm s+a_0$, with $a_0,a_1\in\r$, $a_0\not=0$. In the first case, we distinguish three possibilities:
\begin{enumerate}
\item If $|a_1|<1$, the surface is a rigid motion of the helicoid of second kind. In fact, let $\theta$ such that
$a_1=\sin\theta/\cos\theta$ and define $\alpha(s)=(0,s/\cosh\theta,0)$. We know that $G_{L,h}(\alpha)$ is the helicoid of second kind. On the other hand, we have
$$\phi_t(\gamma(s))=\phi_{t+\theta}(\alpha(s))-(h\theta,0,0)=G_{L,h}(\alpha)-(h\theta,0,0),$$
that is, our surface is a translation of the helicoid of second kind in the direction of the axis $L$.
 \item If $|a_1|=1$, then $W=0$, and this is not possible.
 \item The case $|a_1|>1$ is analogous with the case $|a_1|<0$. The surface is  the helicoid of third kind followed of a translation in the direction of the axis.
 \end{enumerate}

\subsection{The axis is lightlike }

By Proposition \ref{pr2}, we consider that $\gamma$ lies in the plane $<(1,0,1),(0,1,0)>$. If $\gamma$ is not a graph on the $y$-axis, then $\gamma(s)=(s,b,s)$ and the corresponding surface is the parabolic null cylinder with $H=0$.
Assume now $\gamma(s)=(f(s),s,f(s))$.   The mean curvature is
 $$H=-\frac12 \frac{4h^2(f'-2sf'')}{(4h \epsilon (s+hf'^2))^{3/2}}.$$
We assume that $H=0$. Then the numerator is a polynomial on $s$ that must be zero. The degree of this polynomial equation is $m-1$ if $m\geq 1$. The  leader coefficient is  $-4h^2 ma_m(2m-3)$. Then $a_m=0$, which it is a contradiction. If $m=0$, then   $f(s)$ is constant, $f(s)=a_0$, and $H=0$.

Assume that $H$ is a non-zero constant. The polynomial equation  (\ref{h2}) is of degree  $6m-6$ if $m\geq 2$. The coefficient is $-256H^2 h^6 m^6 a_m^6$. Thus $a_m=0$: contradiction. If $m=1$, then (\ref{h2}) is a polynomial equation of degree $3$, whose leader coefficient is $256 h^3 H^2$:  contradiction. Finally, the case $m=0$ leads to $H=0$.

For the lightlike case, the only possibility is that $H=0$ and $f$ is a constant function, $f(s)=a_0$. The surface parametrizes as
$$X(s,t)=(st+ h(\frac{t^3}{t}-t),s+ht^2,st+h(\frac{t^3}{3}+t))+(a_0,0,a_0), \ a_0\in\r,$$
that is, it a translation of the Cayley's surface. This finishes the proof of Theorem \ref{t1}.

We end this section with a result about a result on non-umbilical timelike surfaces that satisfy $H^2-K=0$. Recall that for a non-degenerate surface, we have $H^2-\epsilon K\geq 0$. In the case that the surface is spacelike ($\epsilon=-1$), the Weingarten map $A_p$ is diagonalizable and the equality $H(p)^2+ K(p)= 0$ means that $p$ is umbilical: $A_p=\lambda(p)(\mbox{id})$. However, if the surface is timelike ($\epsilon=1$), $A_p$ could be not diagonalizable. In fact, it could be that $H(p)^2- K(p)= 0$ and $p$ is not umbilical. From Theorem \ref{t1}, the surface with timelike axis generated by $\gamma(s)=(s,0,\pm s+a_0)$ is a surface with $|H|=1/h$ and $K=1/h^2$. Then $H^2-K=0$. On the other hand, the surface with spacelike axis generated by $\gamma(s)=(0,s,\pm s+a_0)$ satisfies
$H=K=0$. In this sense, we show:

\begin{theorem} \label{t4}  Consider a  helicoidal timelike surface in $\e_1^3$   whose generating curve is the graph of a polynomial $f(s)=\sum_{n=0}^m a_n s^n$. If $H^2-K=0$ on the surface, then $m\leq 1$.  Moreover, up a rigid motion of $\e_1^3$, the parametrization is
\begin{enumerate}
\item If the axis is timelike, $X(s,t)=(s\cos{(t)},s\sin{(t)},\pm s+a_0+ht)$. In this case, $a_0\in\r$, $|H|=1/h$ and $K=1/h^2$.
\item If the axis is spacelike,
$$X(s,t)=(ht,(\pm s+a_0)\sinh{(t)}+s\cosh{(t)},(\pm s+a_0)\cosh{(t)}+s\sinh{(t)}).$$
Here $a_0\not=0$ and $H=K=0$.
\item If the axis is lightlike, then the surface is the parabolic null cylinder
$$X(s,t)=(s+bt+h(-t+\frac{t^2}{2}),b+ht^2,s+bt+h(t+\frac{t^2}{3})).$$
Here $b\in\r$ and $H=K=0$.
\end{enumerate}
\end{theorem}

\begin{proof} From (\ref{h1}) and (\ref{hk22}), and the fact that $W<0$,   the identity  $H^2-K=0$ is equivalent to $H_1^2-4WK_1=0$. Because the generating curve is the graph of a polynomial on $s$, this equation writes as
$P(s)=\sum_{n=0}^kA_n s^n=0$. Then all coefficients must be zero. We distinguish the three cases of causal character of the axis:
\begin{enumerate}
\item The axis is timelike. The Lorentzian cylinder does not satisfy $H^2-K=0$. Thus we assume 
that $\gamma(s)=(s,0,f(s))$. The equation $H_1^2-4WK_1=0$ writes as
$$\Big((-2h^2+s^2)f-s^2f'^3+s(s^2-h^2)f''\Big)^2-4(h^2-s^2+s^2f'^2)(h^2-s^3f'f'')=0.$$
If  $m\geq 2$, the leader coefficient of $P$ comes from $s^4f'^6$, which it is $m^6a_m^6$: contradiction. If $m=1$, the degree of $P$ is $k=4$, and the leader coefficient is $A_4=a_1^2(1-a_1^2)^2$. Hence we obtain $a_1=\pm 1$. In this case, $|H|=1/h$, $K=1/h^2$, $W=-h^2$ and the Weingarten map is
$$\left(\begin{array}{ll}-1&0\\-h&-1\end{array}\right).$$
If $m=0$, the surface is the helicoid of first kind, which it does not satisfy $H^2-K=0$.
\item The axis is spacelike.  If $\gamma$ is not a graph on the $y$-axis, then $\gamma(s)=(0,b,s)$, but the surface that generates does not satisfy $H^2-K=0$. We suppose that $\gamma(s)=(0,s,f(s))$. The equation to study is
\begin{eqnarray*}
& & h\Big(sf'-sf'^3+f(f'^2-1)+(h^2-s^2)f''+f^2f''\Big)^2\\
& &+4\Big((f-sf')^2+(h^2-s^2+f^2)(-1+f'^2)\Big)\Big(
hf''(f-sf')-h(f'^2-1)^2\Big)=0.
\end{eqnarray*}
  If $m\geq 2$, the polynomial equation is of degree $k=8m-6$ and the leader coefficient is $-4h m^6 a_m^6$. This gives a contradiction. Assume $m=1$. The degree is now $k=2$ and the leader coefficient is $A_2=4h^2(1-a_1^2)^4$. Then $a_1=\pm 1$. The surface satisfies $H=K=0$ and the Weingarten map is
$$\left(\begin{array}{ll}0&0\\-h&0\end{array}\right).$$
In the case that $m=0$, the equation $P=0$ reduces to $-4s^2+4h^2+a_0^2=0$, which it gives a contradiction.
\item The axis is lightlike. We point out that if $\gamma(s)=(s,b,s)$ the surface is the parabolic null cylinder with $H=K=0$. We assume that $\gamma(s)=(f(s),s,f(s))$. Now $H_1^2-4WK_1=0$ reduces to
$$ (h(f'-2sf'' )^2-4(s+hf'^2)(1+2hf'f'')=0.$$
When $m\geq 2$, the  degree of $P$ is $k=4m-5$ and it comes from $-8h^2f'^3f''$. The leader coefficient is $-8h^2 m^4(m-1)a_m^4$: contradiction. If $m=1$, the equation reduces
to $3h a_1^2+4s=0$, which it leads to a contradiction again. If $m=0$, the equation is $hs=0$: contradiction.
\end{enumerate}
\end{proof}

\begin{remark} The helicoidal surfaces that appear in Theorem \ref{t4} are generated by lightlike straight-lines. Both surfaces are ruled and Theorem 2 in \cite{dk} asserts that if a ruling is lightlike, then $H^2=K$, as it occurs in our situation.
\end{remark}

\begin{remark} The minimal timelike surface $X(s,t)=(s\cos{(t)},s\sin({t}),\pm s+a_0 +ht)$  is different from the three helicoids and the Cayley's  surface. For the choice of $a_0=0$, this surface appears in \cite[Ex. 5.3]{hj}. On the other hand, the two  surfaces that appear in Theorem \ref{t4} are {\it linear} Weingarten surfaces, that is,
they satisfy  a relation of type $aH+bK=c$, with $a,b,c\in\r$.
\end{remark}

\begin{figure}[h]\begin{center}
\includegraphics[width=5cm]{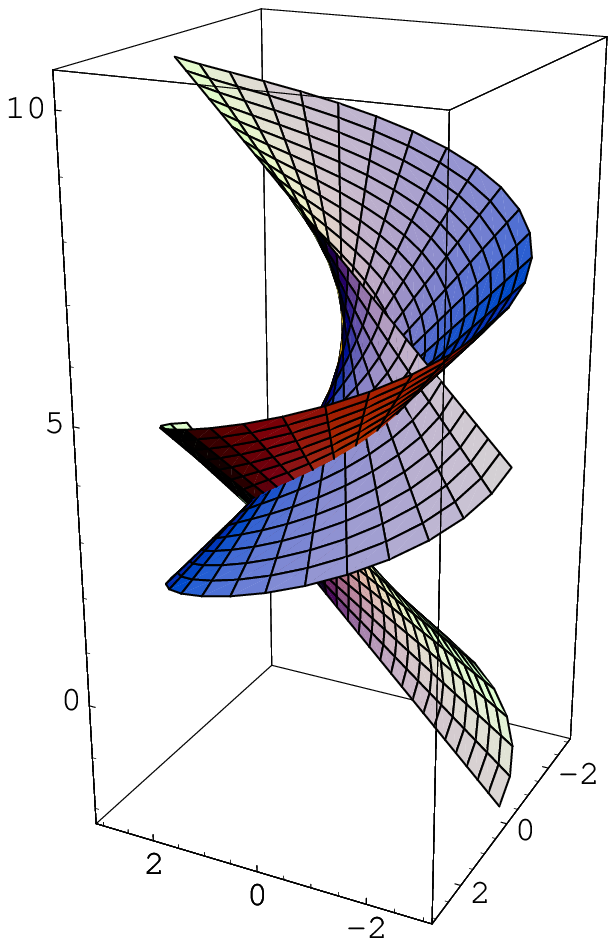}\hspace*{.5cm}\includegraphics[width=4cm]{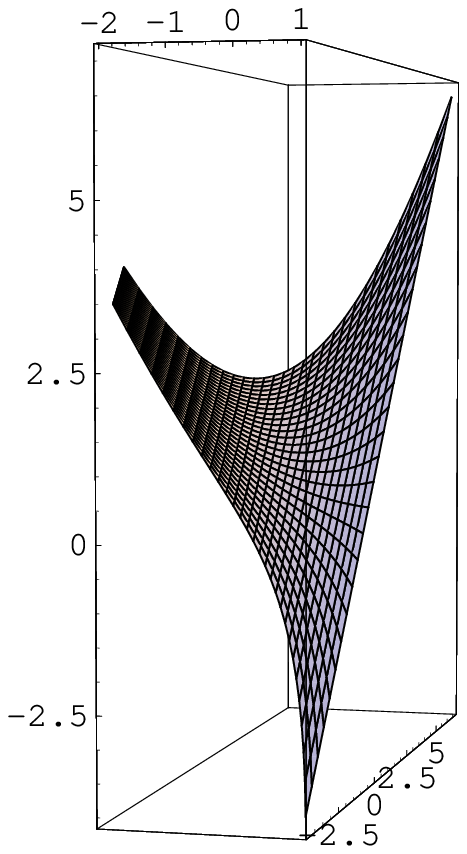}\end{center}
\caption{(left) The timelike $X(s,t)=(s\cos{(t)},s\sin({t}),\pm s+a_0 +ht)$, for $a_0=h=1$. (right) The surface
$X(s,t)=(st+ h(\frac{t^3}{t}-t)+a_0,s+ht^2,st+h(\frac{t^3}{3}+t)+a_0))$ for $a_0=h=1$.}
\end{figure}

\section{Proof of Theorem \ref{t2}}\label{sectt2}

We consider helicoidal surfaces generated by a Lorentzian circle. We distinguish the three cases of causal character of the  axis.

\subsection{The axis is timelike }

Consider the axis $L=<(0,0,1)>$ and the generating curve $\gamma(s)=(x(s),0,z(s))$. Here the circle $\gamma$ lies in the timelike plane $\Pi:\{y=0\}$. Then the parametrization of $\gamma$ is, up a rigid motion of $\Pi$, the curve
$x^2-z^2=\pm r^2$. We take the first possibility, that is, the circle $x^2-z^2=r^2$, because the other one is analogous, namely, $x^2-z^2=-r^2$,  and we left to the reader.  Thus
\begin{eqnarray*}
\gamma(s)&=&\left(\begin{array}{lll}
\cosh{(\theta)}&0&\sinh{(\theta)}\\0&1&0\\\sinh{(\theta)}&0&\cosh{(\theta)}\end{array}\right)
\left(\begin{array}{l}r\cosh{(s)}\\ 0\\ r\sinh{(s)}\end{array}\right)+\left(\begin{array}{l}\lambda\\ 0\\\mu\end{array}\right)\\
&=&(\lambda+r\cosh{(s+\theta)},0,\mu+r\sinh{(s+\theta)}),
\end{eqnarray*}
with $\theta,\lambda,\mu\in\r$. The parametrization of the surface is
$$X(s,t)=(\cos{(t)}(\lambda+r\cosh{(s+\theta)}),\sin{(t)}(\mu+r\cosh{(s+\theta)},\mu+ht+r\sinh{(s+\theta)}).$$
We compute the mean curvature and we separate the cases $H=0$ and $H\not=0$.
\begin{enumerate}
\item If $H=0$, then $H_1=0$. This equation writes
$$\sum_{n=0}^3 A_n\cosh{(n(s+\theta))}=0.$$
Because the functions $\{ \cosh{(n(s+\theta))}\}$ are independent linear, $A_n=0$, for any $n$, $0\leq n\leq 3$, we conclude $A_n=0$. The leader coefficient is $A_3=\frac12r^3(h^2+r^2)$. We obtain then a contradiction.
\item Assume that $H$ is a non-zero constant. We use (\ref{h2}), which writes as
$$\sum_{n=0}^6 A_n\cosh{(n(s+\theta))}=0.$$
Now, 
$$A_6=-\frac18r^6(h^2+r^2)^2(\pm 1+H^2(h^2+r^2)),$$
where $\pm 1+H^2(h^2+r^2)$ depends if the surface is spacelike or timelike. If the choice is $1+H^2(h^2+r^2)$, we get a contradiction. In the case $-1+H^2(h^2+r^2)$, then $H^2=1/(h^2+r^2)$ and $A_5=\lambda r^7(h^2+r^2)/4$. Then
$\lambda=0$. But now $A_2=3h^4 r^6/2$: contradiction.
\end{enumerate}
As conclusion,  the case that the axis is timelike is impossible.

\subsection{The axis is spacelike}

Now  the axis is $L=<(1,0,0)>$ and the generating curve $\gamma(s)=(0,y(s),z(s))$ lies in the plane $\Pi:\{x=0\}$. As in the above case, the plane $\Pi$ is timelike and thus, the Lorentzian circles are rigid motions of the
circle $y^2-z^2=\pm r^2$. As in the case that the axis is timelike, we only consider the case $y^2-z^2=r^2$. Then the generating curve $\gamma$ writes as
\begin{eqnarray*}
\gamma(s)&=&\left(\begin{array}{lll}
1&0&0\\
0&\cosh{(\theta)}&\sinh{(\theta)}\\0&\sinh{(\theta)}&\cosh{(\theta)}\end{array}\right)
\left(\begin{array}{l}0\\ r\cosh{(s)}\\ r\sinh{(s)}\end{array}\right)+\left(\begin{array}{l}0\\ \lambda\\ \mu\end{array}\right)\\
&=&(0,\lambda+r\cosh{(s+\theta)},\mu+r\sinh{(s+\theta)}),
\end{eqnarray*}
with $\theta,\lambda,\mu\in\r$. The parametrization of the surface is now
$$X(s,t)=(ht,\lambda\cosh{(t)}+\mu\sinh{(t)}+r\cosh{(s+t+\theta)}),\lambda\sinh{(t)}+\mu\cosh{(t)}+r \sinh{(s+t+\theta)}).$$
\begin{enumerate}
\item Assume $H=0$. This equation writes as
$$hr^2\Big(-\lambda^2+\mu^2+h^2-r \lambda\cosh{(s+\theta)}+r \mu\sinh{(s+\theta)}\Big)=0.$$
As the functions $\cosh{(s+\theta)}$ and $\sinh{(s+\theta)}$ are independent, we deduce that their coefficients vanish, that is, $\lambda=\mu=0$. Now $H=0$ is equivalent to $h^3r^2=0$: contradiction.
\item If $H$ is a non-zero constant, we use (\ref{h2}). The expression of this equation is
$$\sum_{n=0}^6\Big( A_n\cosh{(n(s+\theta))}+B_n\sinh{(n(s+\theta))}\Big)=0.$$
Thus $A_n=B_n=0$. In particular, and independently if the surface is spacelike or timelike, 
$$A_6=-\frac18(\lambda^2+\mu^2)(\lambda^4+14\lambda^2\mu^2+\mu^4)H^2 r^6=0.$$
Therefore, $\lambda=\mu=0$. Now $P=0$ is $h^6r^4(-1+4H^2r^2)=0$, that is,
$|H|=1/(2r)$. The generating curve is $\gamma(s)=(0,r\cosh{(s+\theta)},r\sinh{(s+\theta)})$, that is,
the circle $y^2-z^2=r^2$ in the plane $\Pi$.
\end{enumerate}
We conclude that if the axis is spacelike, then the generating curve is a Lorentzian circle centered at the axis.

\subsection{The axis is lightlike}

Consider the axis $L=<(1,0,1)>$. A helicoidal surface with axis $L$ has the generating curve in the plane
 $\Pi:x-z=0$. A Lorentzian circle in the plane $\Pi$ is a rigid motion of the circle $s\longmapsto c s(0,1,0)+c/2 s^2(1,0,1)$, $c\not=0$. Then the circle writes as
\begin{eqnarray*}
\gamma(s)&=&\left(\begin{array}{lll}
1-\frac{\theta^2}{2}&\theta &\frac{\theta^2}{2}\\
-\theta &1&\theta\\
-\frac{\theta^2}{2} & t &1+\frac{\theta^2}{2}
\end{array}\right)
\left(\begin{array}{l}c\frac{s^2}{2}\\ cs \\ c\frac{s^2}{2}\end{array}\right)+\left(\begin{array}{l}\lambda\\ \mu\\ \lambda\end{array}\right),
\end{eqnarray*}
with $\theta,\lambda,\mu\in\r$.
\begin{enumerate}
\item Suppose $H=0$. This equation writes as $4c^2h^2(-2\mu+c\theta-chs)=0$. Because this is a polynomial equation on $s$, we get a contradiction.
\item If $H\not=0$, then (\ref{h2}) is a polynomial equation of degree $6$. The leader coefficient is $-256c^6 h^6 H^2$. This gives a contradiction again.
\end{enumerate}

Thus it is not possible that the axis is lightlike. This finishes the proof of Theorem \ref{t2}.

\section{Proof of Theorem \ref{t3}}\label{sectt3}

We do the proof only for the case that the generating curve $\gamma$ is the graph of a polynomial:  when $\gamma$ is a circle, the proof is similar to the given one in Section \ref{sectt2}.

From (\ref{hk22}), we have $K W^2+K_1=0$. Because the generating curve is the graph of a polynomial $f(s)=\sum_{n=0}^m a_n s^n$, we have a polynomial equation on $s$:
\begin{equation}\label{pk}
P(s):=\sum_{n=0}^k A_k s^k=0.
\end{equation}
 Therefore, all coefficients must be zero: $A_n=0$. We distinguish the three cases of axis.
\begin{enumerate}
\item The axis is timelike. If the curve $\gamma$ is not a graph on the $x$-axis, the surface is the 
Lorentzian cylinder $x^2+y^2=r^2$, with $K=0$. Now we use the parametrization (\ref{p-1}) and we have
$$K=\frac{h^2-s^3f'f''}{(-h^2+s^2-s^2f'^2)^2}.$$
If $K=0$ and if $m\geq 2$, the degree of $P$ is $2m$, whose leader coefficient is $-m^2(m-1)a_m^2$. This is not possible. If $m\leq 1$, (\ref{pk}) reduces to $h^2=0$: contradiction.

When $K$ is a non-zero constant, the degree of $P$ is $k=4m$ if $m\geq 2$. The leader coefficient is $m^4 a_m^4 K$: contradiction. If $m=1$, the degree of $P$ is $k=4$, with $A_4=K(1-a_1^2)^2$. Hence we obtain $a_1=\pm 1$. Now $P=h^2(-1+h^2K)=0$ and so, $K=1/h^2$. When $m=0$, the degree of $P$ is $4$ again with $A_4=K$: contradiction.

\item The axis is spacelike. As the helicoid of third kind has not constant Gauss curvature, we use  (\ref{p-2}) and the expression of $K$ is
$$K=\frac{h(h(1-f'^2)^2-h(f-sf')f'')}{ ((f-sf')^2+(h^2-s^2+f^2)(1-f'^2))^2}.$$
If $K=0$ and if $m\geq 2$,  the degree of $P$ is $k=4m-4$ with leader coefficient $-h^2m^4 a_m^4$: contradiction. If $m=1$, then $P=h^2(1-a_1^2)$. Thus $a_1=\pm 1$, and $a_0\not=0$ in order to have $W\not=0$; if $m=0$, $P=-h^2=0$: contradiction.

Assume that $K\not=0$. If $m\geq 2$, the degree of $P$ is $8m-4$ with leader coefficient $m^4 a_m^8 K$: contradiction. When $m=1$, $k=4$ with $A_4=K(1-a_1^2)^4$. Then $a_1=\pm 1$. Now $P$ reduces to
$Ka_0^4=0$. Then $a_0=0$, but now $W=0$: contradiction. When $m=0$, the degree of $P$ is $k=4$   with $A_4=K$: contradiction.

\item The axis is lightlike. We point out that if $\gamma$ is not a graph on the $y$-axis, the corresponding surface is the parabolic null cylinder, which it is a surface with $K=0$. On the other hand, the value of $K$ using  (\ref{p-3}) is
$$K=\frac{1+2hf'f''}{4(s+hf'^2)^2}.$$
Let $K=0$. If $m\geq 2$, the degree of $P$ is $2m-3$ with leader coefficient $-8h^3m^2(m-1)a_m^2$, which it is a contradiction. If $m\leq 1$, $P=-4h^2$: contradiction.

When $K$ is a non-zero constant,  the degree of $P$ is $k=4m-4$ if $m\geq 2$. The leader coefficient is $16h^4 m^4 a_m^4 K$: contradiction. If $m\leq 1$, the degree of $P$ is $k=2$, with $A_2=16h^2 K$: contradiction.
\end{enumerate}


\end{document}